\RequirePackage{amsmath}
\documentclass[runningheads]{llncs}
\usepackage{fancyvrb}
\usepackage[usenames]{color}
\usepackage{amssymb}
\usepackage{amsmath}
\usepackage{amsfonts}
\usepackage{amscd}
\usepackage{graphicx}
\usepackage
[colorlinks=true,
linkcolor=blue,
urlcolor=blue,
filecolor=blue,
citecolor=blue]
{hyperref}

\usepackage{color}
\usepackage{graphics}
\usepackage{latexsym}
\usepackage{epsf}

\usepackage{array}
\newcolumntype{L}[1]{>{\raggedright\let\newline\\\arraybackslash\hspace{0pt}}m{#1}}
\newcolumntype{C}[1]{>{\centering\let\newline\\\arraybackslash\hspace{0pt}}m{#1}}
\newcolumntype{R}[1]{>{\raggedleft\let\newline\\\arraybackslash\hspace{0pt}}m{#1}}

\DeclareMathOperator{\Pal}{Pal}

\begin{document}
\title{Repetitions in infinite palindrome-rich words}
\titlerunning{Repetitions in infinite rich words}
%
\author{
Aseem R. Baranwal \orcidID{0000-0001-5318-6054}
\and
Jeffrey Shallit \orcidID{0000-0003-1197-3820}
}

\authorrunning{A. Baranwal and J. Shallit}
%
\institute{
    School of Computer Science, University of Waterloo\\
    Waterloo, ON N2L 3G1, Canada\\
    \email{aseem.baranwal@uwaterloo.ca}\\
    \email{shallit@uwaterloo.ca}
}
\maketitle              
\begin{abstract}
Rich words are characterized by containing the maximum possible number of distinct palindromes. Several characteristic properties of rich words have been studied; yet the analysis of repetitions in rich words still involves some interesting open problems. We address lower bounds on the repetition threshold of infinite rich words over 2 and 3-letter alphabets, and construct a candidate infinite rich word over the alphabet $\Sigma_2=\{0,1\}$ with a small critical exponent of $2+\sqrt{2}/2$.  This represents the first progress on an open problem of Vesti from 2017.

\end{abstract}
\section{Introduction}
Palindromes---words equal to their reversal---are among the most widely studied repetitions in words. The class of palindrome-rich words, or simply rich words---those words containing the maximum possible number of palindromes---was introduced in the papers \cite{Brlek&Hamel&Nivat&Reutenauer:2004,deLuca&Glen&Zamboni:2008,Glen&Justin&Widmer&Zamboni:2009}. Since then, rich words have received much attention in the combinatorics on words literature; see, for example, \cite{Bucci&DeLuca&Glen&Zamboni:2009,Guo&Shallit&Shur:2016,Vesti:2014}.

\subsection{Preliminaries}
In this section we provide the preliminary definitions and results that we use throughout the paper, along with the motivation behind our work.

\begin{definition}
A finite word $w$ is {\em rich} if it contains $|w|$ distinct nonempty palindromes. An infinite word $w$ is rich if all its factors are rich.
\end{definition}

We say that a word $u=z^e$ has \textit{exponent} $e$ and \textit{period} $p=|z|$, where $e = |u|/p$ is a positive rational number that denotes the number of times $z$ is repeated. We say $u$ is \textit{primitive} if its only integer exponent is $1$.
The word $w$ is an \textit{overlap} if $w=uuu'$ where $u'$ is a prefix of $u$.


\begin{example}
The word $u=00010001$ is rich, because it has 8 distinct nonempty palindromes as factors, while the word $v=00101100$ is not rich. The word $u$ has period 4 and exponent 2, since $u=z^e$, where $z=0001$ and $e=2$.
\end{example}

\begin{definition}
For a given alphabet $\Sigma$, a mapping $\varphi$ on $\Sigma^*$ is an {\em antimorphism} if $\varphi(uv) = \varphi(v)\varphi(u)$ for all $v,w \in \Sigma^*$.
\end{definition}

\begin{definition}
The {\em critical exponent} of an infinite word $w$ is defined to be the supremum of the set of all rational numbers $e$ such that there exists a finite nonempty factor of $w$ with exponent $e$.
\end{definition}

\begin{definition}
The {\em repetition threshold} on an alphabet of size $k$ is the infimum of the set of exponents $e$ such that there exists an infinite word that avoids greater than $e$-powers.
\end{definition}

In other words, the repetition threshold is the smallest possible critical exponent of a word over an alphabet of size $k$. Dejean gave a famous conjecture about this threshold in~\cite{Dejean:1972}, which was proven by Currie and Rampersad~\cite{Currie&Rampersad:2011}, and independently by Rao~\cite{Rao:2011}. The repetition threshold can also be studied for a limited class of infinite words. For example, Rampersad et al. studied this threshold for infinite balanced words in~\cite{Rampersad:2018}. In this paper, we study the repetition threshold $RT(k)$ for infinite rich words over an alphabet of size $k$.

\subsection{Previous work}
Let the word $w$ be the fixed point of a given involutive antimorphism $\Theta$. We say $w$ is a \textit{$\Theta$-palindrome} if $w = \Theta(w)$. The set of $\Theta$-palindromic factors of a word $w$ is denoted by $\Pal_\Theta(w)$. In 2013, Pelantov\'a and Starosta introduced the idea of $\Theta$-palindromic defect.


\begin{definition}
The {\em $\Theta$-palindromic defect} of a finite word $w$, denoted by $D_\Theta(w)$, is defined as
$$D_\Theta(w) = |w| + 1 - \gamma_\Theta(w) - |\Pal_\Theta(w)|,$$
where $\gamma_\Theta(w)$ = $|\big\{\{a, \Theta(a)\}: a \in \Sigma\text{, a occurs in } w\text{ and }a\neq\Theta(a)\big\}|$.
\end{definition}

Further, they proved that all recurrent words with a finite $\Theta$-palindromic defect contain infinitely many overlapping factors~\cite{Pelantova&Starosta:2013}. This result leads to the following theorem \cite{Pelantova&Starosta:2013}.

\begin{theorem}
All infinite rich words contain a square.
\label{one}
\end{theorem}

\autoref{one} provides a lower bound on the repetition threshold for infinite rich words over a $k$-letter alphabet; namely $RT(k) \ge 2$. In~\cite{Vesti:2017}, Vesti gives both upper and lower bounds on the length of the longest square-free rich words, and proposes the open problem of determining the repetition threshold for infinite rich words.

\section{Results over the binary alphabet}
We construct an infinite binary rich word and determine the value of its critical exponent. We further conjecture that this value is the repetition threshold for the binary alphabet, based on supporting evidence from computation. We define the word ${\bf r}$ as the image of a fixed point, ${\bf r}=\tau(\varphi^\omega(0))=001001100100110\cdots$, where the morphisms $\varphi$ and $\tau$ are defined as follows:

\begin{center}
\begin{tabular}{ r l C{3cm} r l }
 $\varphi$: & $0 \to 01$ & & $\tau$: & $0 \to 0$ \\
         & $1 \to 02$ & & & $1 \to 01$ \\
         & $2 \to 022$, & & & $2 \to 011$.
\end{tabular}
\end{center}

\subsection{Automatic theorem-proving}
We utilize the automatic theorem-proving software \texttt{Walnut}, written by Hamoon Mousavi, to constructively decide first-order predicates concerning the word ${\bf r}$ \cite{Mousavi:2016}. To enable \texttt{Walnut} to work with the word ${\bf r}$, we require an automaton with output that produces ${\bf r}$. Computing the lengths $L_i=|\tau(\varphi^i(0))|$ for $i\ge 0$, we note that
$$L_0 = 1,\ L_1 = 3,\ \text{and } L_i = 2L_{i-1}+L_{i-2} \text{ for } i\ge 2.$$
Since the Pell numbers are defined by the recurrence $P_0=0$, $P_1=1$, and $P_n=2P_{n-1}+P_{n-2}$, this suggests that the word ${\bf r}$ is \textit{Pell-automatic}, meaning that there exists an automaton that takes as input an integer $N$ represented in the Pell number system, and outputs the symbol in ${\bf r}$ at index $N$. The Pell number system is a non-standard positional number system in the family of Ostrowski numeration systems~\cite{Ostrowski:1922}. We utilize the Pell adder constructed in~\cite{Baranwal&Shallit:2019} to enable writing predicates in this number system. The \texttt{Walnut} version equipped with the adder is available on \href{https://github.com/aseemrb/Walnut/}{GitHub}.\footnote{Repository: \url{https://github.com/aseemrb/Walnut/} .}

\subsection{Constructing the automaton}
Using the methods of Angluin~\cite{Angluin:1987}, we construct an automaton with output for the word ${\bf r}$.    \autoref{fig:r:unrestricted} represents the automaton. Note that this automaton consists of 4 states, and we have not restricted the Pell representations to be unique for each integer, meaning that the input may end with a 2, and a non-zero digit may follow a 2. The node labels in the figure represent the state and the corresponding output symbol.

\begin{figure}[ht]
    \center{\includegraphics[width=0.8\textwidth]{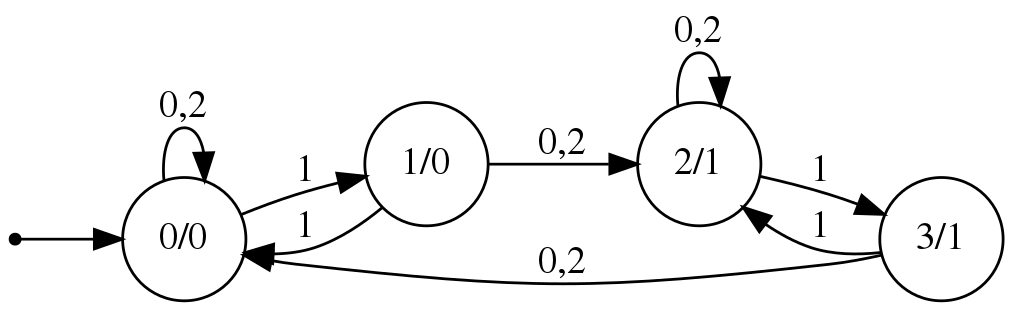}}
    \caption{Automaton for the infinite word ${\bf r}$.}
    \label{fig:r:unrestricted}
\end{figure}

Before we proceed, we prove that this automaton produces the same word as given by $\tau(\varphi(0))$. To do this, we restrict the automaton in \autoref{fig:r:unrestricted} to only consider unique integer representations in the Pell number system. Thus, the least significant digit is $< 2$, and a 2 is always followed by a 0. This gives the automaton in \autoref{fig:r:restricted}, which represents ${\bf r}=g(f^\omega(0))$ for morphisms $f$ and $g$, given by

\begin{center}
\begin{tabular}{ r l C{3cm} r l }
 $f$: & $0 \to 012$ & & $g$: & $0 \to 0$ \\
      & $1 \to 304$ & & & $1 \to 0$ \\
      & $2 \to 0$ & & & $2 \to \epsilon$ \\
      & $3 \to 354$ & & & $3 \to 1$ \\
      & $4 \to 3$ & & & $4 \to \epsilon$ \\
      & $5 \to 032,$ & & & $5 \to 1$.
\end{tabular}
\end{center}

\begin{figure}[ht]
    \center{\includegraphics[width=\textwidth]{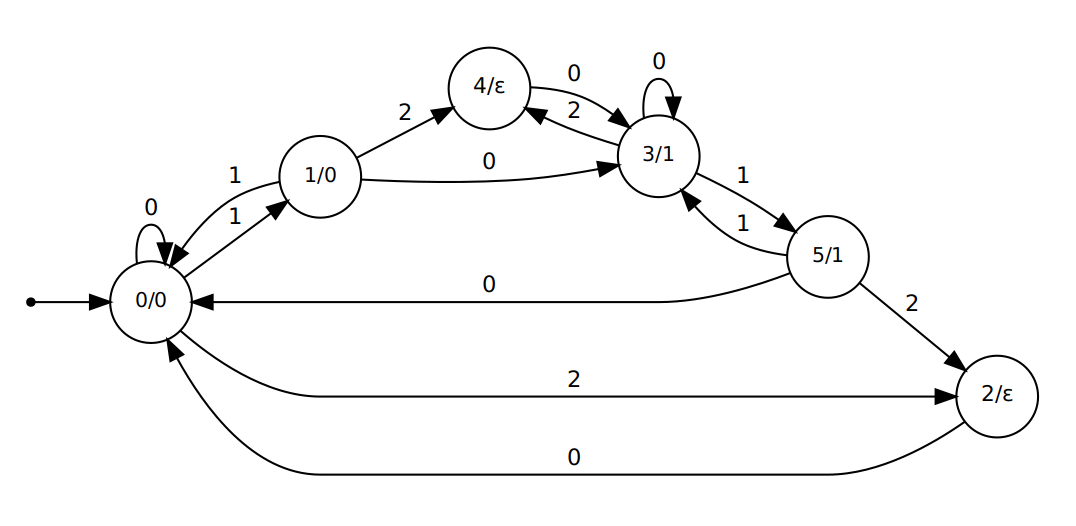}}
    \caption{Restricted automaton for the infinite word ${\bf r}$. Here $\epsilon$ denotes the empty word.}
    \label{fig:r:restricted}
\end{figure}

\subsection{Proof of equivalence of the morphisms}
In this section, we prove that the automaton in \autoref{fig:r:restricted} produces the same infinite word as that produced by morphisms $\varphi$ and $\tau$. We need two lemmas to prove this equivalence.

\begin{lemma}\label{lem:g:f:0}
For all $n\ge 2$, we have $g(f^n(0))=g(f^{n-1}(0))g(f^{n-2}(3))g(f^{n-1}(0))$.
\end{lemma}
\begin{proof}
We prove this by induction on $n$. For $n=2$, we have that
$$g(f^2(0)) = g(f^1(0))g(3)g(f^1(0))=00100.$$
So the base case holds. Next, we construct the induction hypothesis,
$$H_1: g(f^k(0)) = g(f^{k-1}(0))g(f^{k-2}(3))g(f^{k-1}(0)), \forall k \le n.$$
For the inductive step, consider $g(f^{n+1}(0))$. Using the definition of the morphisms $f$ and $g$, we have that,
\begin{align}
  g(f^{n+1}(0)) &= g(f^n(0))g(f^n(1))g(f^n(2)) \nonumber\\
                &= g(f^n(0))g(f^{n-1}(3))g(f^{n-1}(0))g(f^{n-1}(4))g(f^n(2)) \nonumber\\
                &= g(f^n(0))g(f^{n-1}(3))g(f^{n-1}(0))g(f^{n-2}(3))g(f^{n-1}(0)).\label{Hypo:G:F:0}
\end{align}
Using the induction hypothesis $H_1$ in Eq.~\eqref{Hypo:G:F:0}, we get
$$g(f^{n+1}(0)) = g(f^n(0))g(f^{n-1}(3))g(f^n(0)).$$
This completes the proof.
\end{proof}

\begin{lemma}\label{lem:g:f:3}
For all $n\ge 2$, $g(f^n(3)) = g(f^{n-1}(3))g(f^{n-2}(0))g(f^{n-1}(3))$.
\end{lemma}
\begin{proof}
The proof is similar to that of \autoref{lem:g:f:0}, by induction on $n$. For $n=2$, we have
$$g(f^2(3)) = g(f^1(3))g(0)g(f^1(3))=11011.$$
So the base case holds. We have the induction hypothesis,
$$H_2: g(f^k(3)) = g(f^{k-1}(3))g(f^{k-2}(0))g(f^{k-1}(3)), \forall k \le n.$$
For the inductive step, consider $g(f^{n+1}(3))$. Using the definition of the morphisms $f$ and $g$, we have that
\begin{align}
  g(f^{n+1}(3)) &= g(f^n(3))g(f^n(5))g(f^n(4)) \nonumber\\
                &= g(f^n(3))g(f^{n-1}(0))g(f^{n-1}(3))g(f^{n-1}(2))g(f^n(4)) \nonumber\\
                &= g(f^n(3))g(f^{n-1}(0))g(f^{n-1}(3))g(f^{n-2}(0))g(f^{n-1}(3)).\label{Hypo:G:F:3}
\end{align}
Using the induction hypothesis $H_2$ in Eq.~\eqref{Hypo:G:F:3}, we get
$$g(f^{n+1}(3)) = g(f^n(3))g(f^{n-1}(0))g(f^n(3)).$$
This completes the proof.
\end{proof}

Now we prove the following equivalence theorem about the words produced by the automaton in \autoref{fig:r:restricted} and the word given by morphisms $\varphi$ and $\tau$.
\begin{theorem}
The infinite words $\tau(\varphi^\omega(0))$ and $g(f^\omega(0))$ are equal.
\end{theorem}
\begin{proof}
We prove this by a simultaneous induction on $n$ with 3 hypotheses.
\begin{align}
\label{P1}
\tau(\varphi^k(0)) &= g(f^k(0))g(f^{k-1}(3)) \\
\label{P2}
\tau(\varphi^k(1)) &= g(f^k(0))g(f^k(3)) \\
\label{P3}
\tau(\varphi^k(2)) &= g(f^k(0))g(f^{k+1}(3))
\end{align}
The base case $k=1$ can be checked by hand. Assume that the hypotheses hold for $k\le n$. Next, we consider the following inductive steps using the definitions of $\varphi$ and $\tau$.
\begin{align*}
    \tau(\varphi^{n+1}(0)) &=  \tau(\varphi^n(0))\tau(\varphi^n(1)) \\
                        &=  g(f^n(0))g(f^{n-1}(3))g(f^n(0))g(f^n(3)) & \text{using~(\ref{P1},\ref{P2})}\\
                        &=  g(f^{n+1}(0))g(f^n(3)). & \text{using \autoref{lem:g:f:0}}.
\end{align*}
\begin{align*}
    \tau(\varphi^{n+1}(1)) &=  \tau(\varphi^n(0))\tau(\varphi^n(2)) \\
                        &=  g(f^n(0))g(f^{n-1}(3))g(f^n(0))g(f^{n+1}(3)) & \text{using~(\ref{P1},\ref{P3})}\\
                        &=  g(f^{n+1}(0))g(f^{n+1}(3)) & \text{using \autoref{lem:g:f:0}}.
\end{align*}
\begin{align*}
    \tau(\varphi^{n+1}(2)) &=  \tau(\varphi^n(0))\tau(\varphi^n(2))\tau(\varphi^n(2)) \\
                        &=  g(f^n(0))g(f^{n-1}(3))g(f^n(0))g(f^{n+1}(3))g(f^n(0))g(f^{n+1}(3)) \\
                        &=  g(f^{n+1}(0))g(f^{n+2}(3)) \text{\qquad\qquad\qquad\qquad using Lemmas \ref{lem:g:f:0}, \ref{lem:g:f:3}}.
\end{align*}
This proves that the hypotheses are true. From Eq.~\eqref{P1}, we have $\tau(\varphi^k(0))=g(f^k(0))g(f^{k-1}(3))$. Letting $n \rightarrow \infty$, we get  $\tau(\varphi^\omega(0))=g(f^\omega(0))$. This completes the proof.
\end{proof}

\subsection{Proof of palindromic richness}
We claim that the infinite word ${\bf r}=g(f^\omega(0)) = 001001100100110\cdots$ is rich. The proof is carried out using \texttt{Walnut} by constructing a set of predicates based on \autoref{Thm:RichPrefix}, as done in~\cite{Schaeffer&Shallit:2016}. We say that a word $w$ has a \textit{unioccurrent} suffix $s$ if $s$ is not a factor of any proper prefix of $w$.

\begin{theorem}\label{Thm:RichPrefix}
(Glen et al. \cite{Glen&Justin&Widmer&Zamboni:2009})
A word $w$ is rich if and only if every prefix of $w$ has a unioccurrent palindromic suffix.
\end{theorem}

In the following predicates, \texttt{R} denotes the automaton in \autoref{fig:r:restricted}. First, we introduce the fundamental predicates that form the building blocks for verification of the richness property.

\begin{enumerate}
\item The predicate \texttt{FactorEq} takes 3 parameters $i,j,n$ and evaluates to true if the length-$n$ factors of ${\bf r}$ starting at indices $i$ and $j$ are equal.

\item The predicate \texttt{Occurs} takes 4 parameters $i,j,m,n$ and evaluates to true if the length-$m$ factor of ${\bf r}$ starting at index $i$ occurs in the length-$n$ factor starting at index $j$, i.e., $\texttt{R}[i..i+m-1]$ is a factor of $\texttt{R}[j..j+n-1]$.

\item The predicate \texttt{Palindrome} takes 2 parameters $i,n$ and evaluates to true if the length-$n$ factor of ${\bf r}$ starting at index $i$ is a palindrome.
\end{enumerate}

\begin{Verbatim}[numbers=left,xleftmargin=5mm,tabsize=2]
def FactorEq "?msd_pell Ak (k < n) => (R[i + k] = R[j + k])";
def Occurs "?msd_pell (m <= n) &
    (Ek (k + m <= n) & $FactorEq(i, j + k, m))";
def Palindrome "?msd_pell Aj,k ((k < n) & (j + k + 1 = n)) =>
    (R[i + k] = R[i + j])";
\end{Verbatim}

By \autoref{Thm:RichPrefix}, for any finite word to be rich, it is sufficient to check if all its prefixes have a unioccurrent palindromic suffix. We use this property to construct the predicate \texttt{RichFactor} which takes two parameters $i,n$, and evaluates to true if the length-$n$ factor of ${\bf r}$ starting at index $i$ is rich. \autoref{fig:richfactor:predicate} shows the representation of variables in the predicate.

\begin{Verbatim}[numbers=left,xleftmargin=5mm,tabsize=2]
def RichFactor "?msd_pell
    Am ((m >= 1) & (m < n)) =>
        (Ej (i <= j) & (j < i + m) &
         $Palindrome(j, i + m - j) &
         ~$Occurs(j, i, i + m - j, m - 1))";
\end{Verbatim}

\begin{figure}[ht]
    \center{\includegraphics[width=0.8\textwidth]{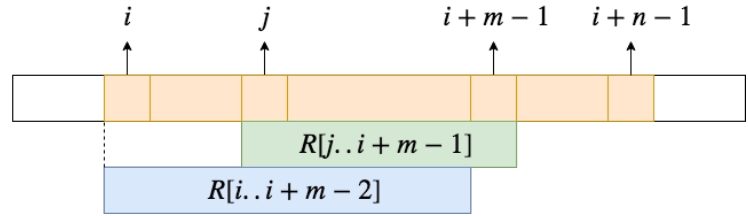}}
    \caption{Representation of variables $i,j,m,n$ in the predicate \texttt{RichFactor}. It evaluates to true if the word $\texttt{R}[i..i+n-1]$ is rich.}
    \label{fig:richfactor:predicate}
\end{figure}

Now, we simply check that all prefixes of ${\bf r}$ are rich to show that the infinite word ${\bf r}$ is rich. The following predicate, \texttt{R\_Is\_Rich} evaluates to true, which completes the proof.
\begin{Verbatim}[numbers=left,xleftmargin=5mm,tabsize=2]
eval R_Is_Rich "?msd_pell An $RichFactor(0, n)";
\end{Verbatim}

\subsection{Determining the critical exponent}
To determine the critical exponent, first, we compute the periods $p$ such that a repetition with exponent $\ge 5/2$ and period $p$ occurs in ${\bf r}$.
\begin{Verbatim}[numbers=left,xleftmargin=5mm,tabsize=2]
eval HighPowPeriods "?msd_pell (p >= 1) &
    (Ei Aj (2*j <= 3*p) => R[i + j] = R[i + j + p])":
\end{Verbatim}
The language accepted by the produced automaton is $0^*1100^*$, which is the Pell-base representation of numbers of the form $P_n + P_{n-1}$, for $n \ge 3$.
Next, we compute pairs of integers $(n, p)$ such that ${\bf r}$ has a factor of length $n+p$ with period $p$, and this factor cannot be extended to a longer factor of length $n+p+1$ with the same period.
\begin{Verbatim}[numbers=left,xleftmargin=5mm,tabsize=2]
def MaximalReps "?msd_pell Ei
    (Aj (j < n) => R[i + j] = R[i + j + p]) &
    (R[i + n] != R[i + n + p])";
\end{Verbatim}
Finally, we compute the pairs $(n, p)$ where $p$ matches the regular expression $0^*1100^*$ in the Pell base representation, and $n+p$ is the maximum possible length of any factor with period $p$.
\begin{Verbatim}[numbers=left,xleftmargin=5mm,tabsize=2]
eval HighestPowers "?msd_pell
    $HighPowPeriods(p) &
    $MaximalReps(n, p) &
    (Am $MaximalReps(m, p) => m <= n)";
\end{Verbatim}

\begin{figure}[ht]
    \center{\includegraphics[width=\textwidth]{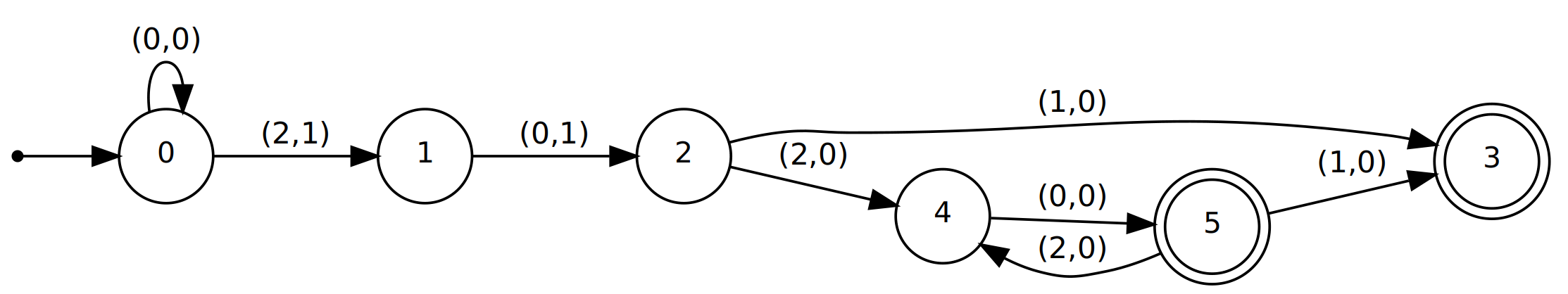}}
    \caption{Pairs $(n,p)$ satisfying the predicate \texttt{HighestPowers}.}
    \label{highestpow}
\end{figure}

\autoref{highestpow} shows the automaton produced by the predicate \texttt{HighestPowers}. It accepts pairs $(n,p)$ of the following forms:
\begin{align}
\label{case1}
& \binom{0}{0}^* \binom{2}{1} \binom{0}{1} \binom{1}{0},\\
\label{case2}
& \binom{0}{0}^* \binom{2}{1} \binom{0}{1} \binom{2}{0} \binom{0}{0} \bigg \{ \binom{2}{0} \binom{0}{0} \bigg \}^*,\text{ or}\\
\label{case3}
& \binom{0}{0}^* \binom{2}{1} \binom{0}{1} \binom{2}{0} \binom{0}{0} \bigg \{ \binom{2}{0} \binom{0}{0} \bigg \}^* \binom{1}{0}.
\end{align}

Here, the length of the words is $l=n+p$ and the period is $p$.  Eq.~\eqref{case1} corresponds to $n=(201)_P=11$ and $p=(110)_P=7$. Thus we have
$$e = \frac{l}{p} = \frac{n+p}{p} = \frac{18}{7} \approx 2.57.$$
Eq.~\eqref{case2} corresponds to
$$n=\sum_{1 \le i \le k}2P_{2k} = P_{2k+1}-1, \text{\ \ } p=P_{2k} + P_{2k-1}.$$
Eq.~\eqref{case3} corresponds to
$$n=1+\sum_{1 \le i \le k}2P_{2k+1} = P_{2k+2}-1, \text{\ \ } p=P_{2k+1} + P_{2k}.$$

Putting $m=2k-1$ for (\ref{case2}), and $m=2k$ for (\ref{case3}), we notice that the expressions for $n$ and $p$ coincide.
\begin{align*}
  e &= \frac{P_{m+2}+P_{m+1}+P_m-1}{P_{m+1}+P_m} \\
  &= 2 + \frac{P_{m+1}-1}{P_{m+1}+P_m}.
\end{align*}
Since Pell numbers are the convergents of $\sqrt{2}-1$, and the ratio $P_{m+1}/P_m$ converges to $\sqrt{2}+1$, we have that
\begin{align}
e & =2 + \frac{P_{m+1}-1}{P_{m+1}+P_m} \nonumber\\
\label{eqn:critexp}
& <2+\frac{\sqrt{2}+1+1/P_{m}^2-1/P_{m}}{\sqrt{2}+2-1/P_{m}^2} \ .
\end{align}

For $m \ge 4$, as $m\rightarrow\infty$, the value in Eq.~\eqref{eqn:critexp} is increasing, and tends to $2+\sqrt{2}/2$. Thus, the critical exponent of the word ${\bf r}$ is $2 + \sqrt{2}/2$. The \texttt{Walnut} commands for verifying richness and computing the critical exponent are available on \href{https://github.com/aseemrb/Walnut/blob/master/Command Files/rich2.txt}{GitHub}.\footnote{URL: \url{https://github.com/aseemrb/Walnut/blob/master/Command Files/rich2.txt} .}

\subsection{Optimality of the critical exponent}
A backtracking computation shows that the longest rich binary word with critical exponent $< 2.700$ is of length 1339. Combining this with the result above, we obtain the following bounds.
$$2.700\le RT(2) \le 2 + \frac{\sqrt{2}}{2}=2.7071\ldots$$

\subsection{Larger alphabets}
For an alphabet of size $k=3$, backtracking search shows that $RT(3) \ge 9/4$. The longest word that has a critical exponent $<9/4$ is of length 114. For $k=4$ and the exponent threshold $11/5$, our search program has reached words of length 3800 and has not terminated.

\section{Faster backtracking}\label{faster:backtrack}

In this section, we discuss some methods to optimize our backtracking algorithm. The most obvious optimization is to consider the following.
\begin{enumerate}
    \item Without loss of generality, we assume that the word starts with a 0.
    \item We impose the restriction that the first occurrence of the symbol $a$ occurs before the first occurrence of symbol $b$ if $a < b$.
\end{enumerate}

\subsection{Lyndon method}
Since our goal is to check if there is an infinite rich word with critical exponent less than a preset threshold, we can utilize the Lyndon method to prune certain branches of the backtracking search tree. A Lyndon word is a primitive nonempty word that is strictly smaller in lexicographic order than all of its rotations. If a word satisfies the properties of richness and the critical exponent being less than some threshold, then all factors of the word also satisfy these properties. This fact helps us by pruning those paths in the search tree that lead to a suffix that is lexicographically smaller than the word itself.

\subsection{Counting palindromes}
To check for richness, Groult et al. give a linear time algorithm to count the number of distinct palindromes in a word~\cite{Groult&Prieur&Richomme:2010}. Their algorithm is based on two major ideas: a linear-time algorithm by Gusfield to compute all maximal palindromes in a word~\cite{Gusfield:1997}, and a linear-time algorithm by Crochermore and Ilie to compute the LPF (longest previous factor) array~\cite{Crochemore:2008}. However, their approach is not helpful to our problem since it requires linear pre-processing time.

What we require is a fast online algorithm such that given the number of distinct palindromes for a word $w$ over an alphabet $\Sigma$, we can find the number of distinct palindromes in the word $wa$ for all $a\in \Sigma$ in constant amortized time. Such an algorithm is given by Rubinchik and Shur~\cite{Rubinchik&Shur:2016}. Their primary idea is to construct a graph where each node represents a unique palindrome.
There are two types of edges in this graph:
\begin{enumerate}
    \item \textbf{Border edge}: This is a directed edge from $p$ to $q$ labeled $a$, if $q=apa$ for some $a\in \Sigma$.
    \item \textbf{Suffix edge}: This is an unlabeled directed edge from $p$ to $q$, if $q$ is the longest proper palindromic suffix of $p$.
\end{enumerate}
Whenever we append a new symbol to an already processed word, it takes amortized constant time to maintain this graph. The \texttt{C++} implementation of the algorithm can be found on \href{https://github.com/aseemrb/research-scripts/blob/master/scripts/palin.cpp}{GitHub}.\footnote{URL: \url{https://github.com/aseemrb/research-scripts/blob/master/scripts/palin.cpp} .}

\begin{figure}[ht]
    \center{\includegraphics[width=0.6\textwidth]{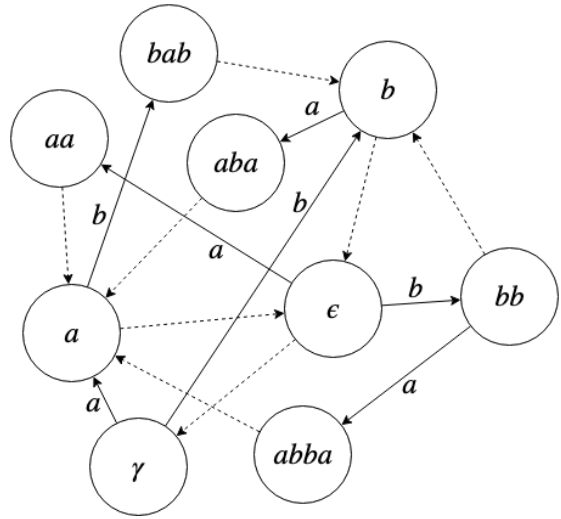}}
    \caption{The graph of palindromes for the word $w=aababba$. Here $\epsilon$ is the empty word and $\gamma$ is the imaginary palindrome word of length $-1$~\cite{Rubinchik&Shur:2016}.}
    \label{fig:palindrome:graph}
\end{figure}

\begin{example}
Figure \ref{fig:palindrome:graph} shows the graph construction for the rich word $aababba$. The number of nonempty palindromes is equal to 7. Note that we have an imaginary word $\gamma$ that has length $-1$ and is a palindrome. The suffix edges are shown by dashed lines, while the border edges are shown with solid lines having labels. We say that a palindrome consisting of a single symbol borders $\gamma$, which makes the implementation of the algorithm easy.
\end{example}

\subsection{Computing maximal runs}
In~\cite{Chen:2007}, Chen et al. present a survey of fast space-efficient algorithms for computing all maximal runs in a string. They also propose some new and faster algorithms for the same. In future work, we aim to understand and implement these algorithms in our backtracking search, so that we are able to compute tighter lower bounds on the repetition threshold more efficiently.

\section{Future prospects}
An obvious direction for further research is to develop novel ideas and methods that may help us prove lower bounds on the repetition threshold of infinite rich words. Another possible direction is to construct infinite rich words over larger alphabets that may serve as candidates for the repetition threshold.


%
%
\bibliographystyle{splncs04}
\bibliography{abbrevs,references}

\end{document}